\documentclass[reqno,11pt]{amsart}
\usepackage{amsmath,amssymb,amsfonts,amsthm, amscd,indentfirst}
\usepackage{amsmath,latexsym,amssymb,amsmath,
	amscd,amsthm,amsxtra}
\usepackage{color}
\usepackage{pdfpages}
\usepackage{graphics}
\usepackage{enumitem}

\usepackage[usenames,dvipsnames]{xcolor}

\usepackage{varioref}
\usepackage{hyperref}
\usepackage[noabbrev,capitalise,nameinlink]{cleveref}

\hypersetup{colorlinks={true},linkcolor={blue},citecolor=blue}
\usepackage[english]{babel}
\usepackage{csquotes}

\usepackage[backend=biber,
style=numeric,datamodel=ext-eprint,maxbibnames=99]{biblatex}
\DeclareSourcemap{
	\maps[datatype=bibtex]{
		\map{
			\step[fieldsource=pmid, fieldtarget=pubmed]
		}
	}
}
\makeatletter
\DeclareFieldFormat{arxiv}{%
	arXiv\addcolon\space
	\ifhyperref
	{\href{http://arxiv.org/\abx@arxivpath/#1}{%
			\nolinkurl{#1}%
			\iffieldundef{arxivclass}
			{}
			{\addspace\texttt{\mkbibbrackets{\thefield{arxivclass}}}}}}
	{\nolinkurl{#1}
		\iffieldundef{arxivclass}
		{}
		{\addspace\texttt{\mkbibbrackets{\thefield{arxivclass}}}}}}
\makeatother
\DeclareFieldFormat{pmcid}{%
	PMCID\addcolon\space
	\ifhyperref
	{\href{http://www.ncbi.nlm.nih.gov/pmc/articles/#1}{\nolinkurl{#1}}}
	{\nolinkurl{#1}}}
\DeclareFieldFormat{mr}{%
	MR\addcolon\space
	\ifhyperref
	{\href{http://www.ams.org/mathscinet-getitem?mr=MR#1}{\nolinkurl{#1}}}
	{\nolinkurl{#1}}}
\DeclareFieldFormat{zbl}{%
	Zbl\addcolon\space
	\ifhyperref
	{\href{http://zbmath.org/?q=an:#1}{\nolinkurl{#1}}}
	{\nolinkurl{#1}}}
\DeclareFieldAlias{jstor}{eprint:jstor}
\DeclareFieldAlias{hdl}{eprint:hdl}
\DeclareFieldAlias{pubmed}{eprint:pubmed}
\DeclareFieldAlias{googlebooks}{eprint:googlebooks}

\renewbibmacro*{eprint}{%
	\printfield{arxiv}%
	\newunit\newblock
	\printfield{jstor}%
	\newunit\newblock
	\printfield{mr}%
	\newunit\newblock
	\printfield{zbl}%
	\newunit\newblock
	\printfield{hdl}%
	\newunit\newblock
	\printfield{pubmed}%
	\newunit\newblock
	\printfield{pmcid}%
	\newunit\newblock
	\printfield{googlebooks}%
	\newunit\newblock
	\iffieldundef{eprinttype}
	{\printfield{eprint}}
	{\printfield[eprint:\strfield{eprinttype}]{eprint}}}

\usepackage{graphicx}
\usepackage{caption}

\usepackage{epsfig,here}
\usepackage{subfigure,here}

\newtheorem{theorem}{Theorem}[section]
\newtheorem{lemma}[theorem]{Lemma}
\newtheorem{proposition}[theorem]{Proposition}
\newtheorem{corollary}[theorem]{Corollary}
\newtheorem{remark}[theorem]{Remark}

\def\rr{\mathbb{R}}

\def\o{\omega}
\def\O{\Omega}
\def\p{\partial}

\def\g{\gamma}

\def\p{\partial}

\def\S{{\Sigma}}
\def\<{\langle}
\def\>{\rangle}
\def\div{{\rm div}}

\def\R{{\mathbb R}}

\def\om{\omega}
\def\Om{\Omega}

\def\p{\partial}

\newcommand{\mbR}{\mathbb{R}}
\newcommand{\mbS}{\mathbb{S}}

\newcommand{\mcH}{\mathcal{H}}

\newcommand{\mcP}{\mathcal{P}}

\newcommand{\mcW}{{\mathcal{W}}}
\newcommand{\ra}{\rightarrow}

\numberwithin{equation} {section}

\begin{document}
	
	\title[Alexandrov's theorem ]{Alexandrov's theorem for anisotropic capillary hypersurfaces in the half-space}

		\author{Xiaohan Jia}
	\address{School of Mathematical Sciences\\
		Xiamen University\\
		361005, Xiamen, P.R. China}
	\email{jiaxiaohan@xmu.edu.cn}
	
		\author{Guofang Wang}
	\address{Mathematisches Institut\\
		Universit\"at Freiburg\\
	Ernst-Zermelo-Str.1,
		79104, Freiburg, Germany}
	\email{guofang.wang@math.uni-freiburg.de}
	
	\author{Chao Xia}
	\address{School of Mathematical Sciences\\
		Xiamen University\\
		361005, Xiamen, P.R. China}
	\email{chaoxia@xmu.edu.cn}
	\author{Xuwen Zhang}
	\address{School of Mathematical Sciences\\
		Xiamen University\\
		361005, Xiamen, P.R. China}
	\email{xuwenzhang@stu.xmu.edu.cn}
	\thanks{This work is  supported by the National Natural Science Foundation of China of  (Grant No. 11871406, 12271449, 12126102)}
	
	\begin{abstract}
	In this paper, we show that any embedded capillary hypersurface in the half-space with anisotropic constant mean curvature is a truncated Wulff shape. This extends Wente's result \cite{Wente80} to the anisotropic case and  He-Li-Ma-Ge's result \cite{HLMG09} to the capillary boundary case. The main ingredients in the proof are a new Heintze-Karcher inequality and a new Minkowski  formula, which have their own inetrest.
		
		\
		
		
	\end{abstract}
	
	\maketitle
	
	\medskip
	
\section{Introduction}

Capillary phenomena appear in the study of the equilibrium shape of liquid drops and crystals in a given solid container. The mathematical model has been established through the work of Young, Laplace,  Gauss and others,  as a variational problem on minimizing a free energy functional under a  volume constraint. A modern formulation of Gauss' model includes a possibly anisotropic surface tension density, which we are interested in. 
For more detailed description on the isotropic and anisotropic capillary phenomena, we refer to \cite{Finn86} and \cite{DM15}.

For our purpose, we consider the anisotropic capillary problem in the half-space$$\mathbb{R}^{n+1}_+=\{x\in \mathbb{R}^{n+1}: \<x, E_{n+1}\>> 0\}.$$Here $E_{n+1}$ denotes the $(n+1)$-coordinate unit vector.
Let $\Sigma$ be a compact orientable embedded hypersurface in $\bar{\mathbb{R}}_+^{n+1}$ with  boundary $\partial\Sigma$ lying on $\p{\mathbb{R}}_+^{n+1}$, which, together with $\p{\mathbb{R}}_+^{n+1}$, encloses a bounded domain $\O$. Let   $\nu$ be the unit  normal of $\Sigma$ pointing outward $\O$.
We consider the free energy functional
\begin{align*}
    \mathcal{E}(\S)=\int_{\Sigma}  F(\nu) dA+\omega_0|\p\Om\cap\p\mbR^{n+1}_+|,
\end{align*}
where the term $\int_{\Sigma}  F(\nu) dA$ is the anisotropic surface tension and the term $\omega_0|\p\Om\cap\p\mbR^{n+1}_+|$ is the wetting energy accounting for the adhesion between the fluid and the walls of the container.
Here $F: \mathbb{S}^n\to \mathbb{R}_+$ is a $C^2$ positive function on $\mathbb{S}^n$ such that $(\nabla^2 F+F \sigma)>0$,
where $\sigma$ is the canonical metric  on $\mbS^n$ and $\nabla^2$ is the Hessian on $\mbS^n$, and $\omega_0\in \mathbb{R}$ is a given constant. The Cahn-Hoffman map associated with $F$ is given by
\begin{eqnarray*}
\Phi:\mathbb{S}^n\to\mathbb{R}^{n+1},\quad \Phi(x)= \nabla F(x)+F(x)x.
\end{eqnarray*}
where $\nabla$ denotes the gradient on $\mbS^n$. The image $\Phi(\mbS^n)$ of $\Phi$ is a strictly convex, closed hypersurface in $\mbR^{n+1}$, which is the unit Wulff shape with respect to $F$, which we denote by $\mathcal{W}_F$. 

In the isotropic case $F\equiv 1$, the global minimizer of $\mathcal{E}$ under a volume constraint is characterized as a spherical cap by De Giorgi, which is the solution to the relative isoperimetric problem, see for example \cite[Chapter 19]{Mag12}. In the anisotropic case, the global minimizer of $\mathcal{E}$ under volume constraint has been characterized by Winterbottom \cite{Win67} as a truncated Wulff shape, which is also called a Winterbottom shape, or Winterbottom construction in applied mathematics, especially in material science, see for example \cite{Bao} and references therein. The Winterbottom construction can be viewed as the capillary counterpart of Wulff construction, which characterizes the global minimizer for purely anisotropic surface tension, see \cite{Wulff, T78, FM91}.
For anisotropic free energy functionals involving a gravitational potential energy term,  
the existence, the regularity and  boundary regularity of global minimizers have been studied  by De Giorgi \cite{D61}, Almgren, \cite{A68} and Taylor \cite{T77}. See also the recent work by De Philippis and Maggi \cite{DM15, DM17}.
For 
the symmetry and uniqueness of global minimizers  we refer to  the work of Baer \cite{Ba15} for a class of $F$ with certain symmetry, following the work of Gonzalez \cite{Gon76} in the isotropic case, via a symmetrization technique. 

In this paper, we shall study the rigidity for the stationary surfaces for the free energy functional $\mathcal{E}$ under a volume constraint.
Given a variation $\{\S_t\}$ of $\S$, whose boundary $\p\S_t$ moves freely on $\p\rr^{n+1}_+$ and according to a variational vector field $Y$ such that $Y|_{\p\S}\in T(\p\rr^{n+1}_+)$, the first variation formula of $\mathcal{E}$  is given by
$$\frac{d}{dt}\Big|_{t=0}\mathcal{E}(\S_t)=\int_\S H^F\left<Y,\nu\right>dA+\int_{\p\S}\left<Y,R(p(\Phi(\nu)))\right> ds,$$
where $H^F$ is the anisotropic mean curvature of $\S$,  $p$ is the projection onto the $\{\nu, E_{n+1}\}$-plane and $R$ is the $\pi/2$-rotation in the $\{\nu, E_{n+1}\}$-plane, see \cite{Pal98, KP06, Koi23}.
It follows that 
the stationary points of $\mathcal{E}$ among $C^2$ hypersurfaces under a volume constraint are anisotropic $\omega_0$-capillary hypersurfaces with constant anisotropic mean curvature. In this paper we say a hypersurface in $\R^{n+1}_+$ with boundary $\p\S\subset \p \R^{n+1}_+$  {\it anisotropic $\omega_0$-capillary} if 
\begin{align}\label{acbc}
    \<\Phi(\nu),-E_{n+1}\>=\omega_0,\quad\text{ on }\p\S.
\end{align}
We emphasize that it is not necessary a constant anisotropic mean curvature hypersurface. Moreover  we are interested in hypersurfaces which  intersect  with $\p\rr^{n+1}_+$ transversely. 

The rigidity of embedded CMC closed hypersurfaces was obtained by Alexandrov \cite{Aleksandrov62} in 
the celebrated Alexandrov's theorem, that any embedded closed hypersurface  of constant mean curvature in $\rr^{n+1}$ must be a sphere. In the proof he introduced the famous  moving plane method.
Wente \cite{Wente80} showed that any embedded compact hypersurface of constant mean curvature with capillary boundary in $\rr^{n+1}_+$ is a spherical cap.
Taking into account of the anisotropy, He-Li-Ma-Ge \cite{HLMG09} proved that any embedded closed hypersurface in $\rr^{n+1}$ with constant anisotropic mean curvature must be a Wulff shape. See also a related result by Morgan \cite{Mor05} in $\rr^2$ for a more general anisotropic function $F$. For the closely related work  on the stability problem of constant anisotropic mean curvature hypersurface without boundary or with capillary boundary, we refer to \cite{Pal98, HL08a, KP06, KP07a, KP07b, KP12, Koi23} and references therein.

Our main result in this paper is the following Alexandrov type theorem for embedded anisotropic capillary hypersurfaces of constant anisotropic mean curvature in $\rr^{n+1}_+$.
\begin{theorem}\label{Thm-Alexandrov0}
		Let $\om_0\in\left(-F(E_{n+1}),F(-E_{n+1})\right)$. Let $\S\subset \bar \rr^{n+1}_+$ be a $C^2$ embedded compact anisotropic $\om_0$-capillary hypersurface with constant anisotropic mean curvature. Then
	$\S$ is an $\om_0$-Wulff shape.
\end{theorem}
An $\om_0$-Wulff shape is part of a Wulff shape in $\bar \rr^{n+1}_+$ such that the anisotropic  capillary boundary condition \eqref{acbc} holds.
We remark that the assumption $\om_0\in\left(-F(E_{n+1}),F(-E_{n+1})\right)$  is a necessary condition so that Wulff shapes intersect with $\p\rr_+^{n+1}$ transversely, see \cref{rem-necc}.

As mentioned above, \cref{Thm-Alexandrov0} for the isotropic case was proved by Wente in \cite{Wente80}, where he used Alexandrov's moving plane method. However, the moving plane method fails in general for the anisotropic case, at least if $F$ has less symmetry. A new proof of Wente's result has been done by the authors \cite{JWXZ22} through the establishment of a Heintze-Kacher-type inequality in the capillary problem, which is inspired by the original idea of Heintze-Karcher \cite{HK78} (see also Montiel-Ros \cite{MR91}). This method is flexible to the anisotropic case and this is the way we achieve \cref{Thm-Alexandrov0}. 

Following this way we first need to establish a Heintze-Harcher type inequality for anisotropic capillary hypersurfaces.
In order to state the inequality, we need a constant vector $E^F_{n+1}\in\mathbb{R}^{n+1}$ defined as
\begin{equation}\label{eq-E-F}
    E^F_{n+1}=\left\{
    \begin{array}{lrl}  \frac{\Phi(E_{n+1})}{F(E_{n+1})}, \quad &\text{if } \om_0<0,\\
    -\frac{\Phi(-E_{n+1})}{F(-E_{n+1})}, \quad &\text{if } \om_0>0.
  %
  \end{array}
  \right.
\end{equation} When $\om_0=0$, one can define it by any unit vector. This constant vector plays a crucial role in the paper.
A hypersurface is said to be strictly anisotropic-mean convex if $H^F>0$. 
Now  we state our anisotropic Heintze-Karcher inequality.
\begin{theorem}\label{Thm-HK}
   Let $\om_0\in\left(-F(E_{n+1}),F(-E_{n+1})\right)$ 
  and  $\S\subset\bar\mbR_+^{n+1}$ be a $C^2$ compact embedded strictly anisotropic-mean convex hypersurface with boundary $\p\S\subset \p\mbR_+^{n+1}$ such that 
  \begin{align}\label{capi-ass}
  \<\Phi(\nu(x)), -E_{n+1}\>=\om(x)\le \om_0,\quad  \hbox{ for  any }x\in \partial\Sigma.\end{align} Then it holds
    \begin{align}\label{eq-HK0}
     \int_\S\frac{F(\nu)+\om_0\left<\nu,E^F_{n+1}\right>}{H^F}dA\geq\frac{n+1}{n}\vert\O\vert.
    \end{align}
    Equality in \eqref{eq-HK0} holds if and only if $\S$ is an $\om_0$-capillary Wulff shape. 
\end{theorem}
We will follow the argument in \cite{JWXZ22} to prove \cref{Thm-HK}. The main idea is to define suitable parallel hypersurfaces $\zeta_F(\cdot, t)$, in order to sweepout the enclosed domain $\O$ and use the area formula to compute the volume. A crucial ingredient is an anisotropic angel comparison principle in \cref{Prop-angle-compare} which enables us to prove the surjectivity of $\zeta_F$.

Then we need the following anisotropic Minkowski type formula.
\begin{theorem}
   \label{Prop-Minkowski0}
 Let $\om_0\in\left(-F(E_{n+1}),F(-E_{n+1})\right)$ 
  and $\S\subset\mbR^{n+1}_+$ be a $C^2$ anisotropic $\om_0$-capillary hypersurface. Let $H_r^F$ be the (normalized) anisotropic $r$-th mean curvature for some $r\in \{1,\ldots,n\}$ and $H_0^F\equiv1$ by convention. Then it holds
\begin{align}\label{formu-Minkowski}
    \int_\S H^F_{r-1}\left(F(\nu)+\om_0\left<\nu, E^{F}_{n+1}\right>\right)-H^F_r\left<x,\nu\right>dA=0.
\end{align}    
In particular,
\begin{align}\label{formu-Minkowski_1}
    \int_\S \left(F(\nu)+\om_0\left<\nu, E^{F}_{n+1}\right>\right)-H^F_1\left<x,\nu\right>dA=0.
\end{align} 
\end{theorem}

\begin{remark}
\normalfont We remark the importance of using the constant vector $E^F_{n+1}$.
In fact, \eqref{eq-HK0} and \eqref{formu-Minkowski} hold true, if we  replace $E^F_{n+1}$ by $E_{n+1}$. However, if one uses $E$ instead of $E^F,$ we  could only prove our main Theorem , \cref{Thm-Alexandrov0},  for a smaller range $
\om_0\in (-1/F^o(E_{n+1}), 1/F^o(-E_{n=1}))$. For one of reasons see \cref{Prop-non-negative}. The main reason lies in the proof of 
the Heintze-Karcher inequality. 
For details, see \cref{rem-1} and \cref{rem-2}. This is one of crucial differences 
between the isotropic case and the anisotropic case.
\end{remark}




For the isotropic case $F\equiv 1$, \eqref{eq-HK0} and \eqref{formu-Minkowski} were proved by the authors \cite{JWXZ22}. We refer to  \cite{JWXZ22} and \cite{WX22} for a historical description of the Heintze-Karcher inequality and the Minkowski formula respectively, and references therein.

The anisotropic Heintze-Karcher inequality and the anisotropic Minkowski formula for closed hypersurfaces have been proved by He-Li-Ma-Ge \cite[Theorem 4.4]{HLMG09} and He-Li \cite{HL08b}. In this case, our argument provides a slight improvement for the anisotropic Heintze-Karcher inequality.

\begin{corollary}\label{coro1}
   Let $\S\subset\bar\mbR^{n+1}$ be a $C^2$ closed embedded strictly anisotropic-mean convex hypersurface. Then it holds
    \begin{align}\label{eq1.6}
     \int_\S\frac{F(\nu)}{H^F}dA\geq
     \frac{n+1}{n}\vert\O\vert + \max\left\{0, \max_{e\in \mathbb{S}^n} \int_\S \frac{ \langle \nu,\Phi (e) \rangle }{H^F}\, dA \right\}.
    \end{align}
    Equality holds if and only if $\S$ is a Wulff shape. 
\end{corollary}

Finally we follow an argument of Ros \cite{Ros87} by combining \cref{Thm-HK} and \cref{Prop-Minkowski0} to establish  the Alexandrov type theorem for capillary hypersurfaces with constant anisotropic mean curvature, \cref{Thm-Alexandrov0}, and also the Alexandrov type theorem for capillary hypersurfaces with
constant higher order anisotropic mean curvature whose definition will be given in \cref{Sec-2}.
\begin{theorem}\label{Thm-Alexandrov1}
		Let $\om_0\in\left(-F(E_{n+1}),F(-E_{n+1})\right)$. Let $\S\subset \bar \rr^{n+1}_+$ be a $C^2$ embedded compact anisotropic $\om_0$-capillary hypersurface with constant $r$-th anisotropic mean curvature for some $r\in \{2,\ldots,n\}$. Then
	$\S$ is an $\om_0$-Wulff shape.
\end{theorem}


The rest of the paper is organized as follows.
In \cref{Sec-2}, we provide more details about the anisotropic mean curvature and
the higher order anisotropic mean curvature, together with 
the Wulff shape and the truncated  Wulff shape, 
the $\om_0$-Wulff shape.
In \cref{Sec-3}, we prove the Minkowski type formula in \cref{Prop-Minkowski0} and the Heintze-Karcher type inequality in \cref{Thm-HK}. In \cref{Sec-4}, we prove the Alexandrov type theorem, \cref{Thm-Alexandrov0} and \cref{Thm-Alexandrov1}.

\

\section{Preliminaries}\label{Sec-2}

Let $F: \mathbb{S}^n\to \mathbb{R}_+$ be a $C^2$ positive function on $\mathbb{S}^n$ such that $(\nabla^2 F+F \sigma)>0.$
We denote $$A_F=\nabla^2 F+F\sigma.$$
Let $F^o:\mbR^{n+1}\ra\mbR$ be defined by
\begin{align*}
    F^o(x)=\sup\left\{\frac{\left<x,z\right>}{F(z)}\Big| z\in\mbS^n\right\},
\end{align*}
where $\left<\cdot,\cdot\right>$
 denotes the standard Euclidean inner product.
We collect some well-known facts on $F$ and $F^o$, see e.g. \cite{HLMG09}.
\begin{proposition}\label{basic-F}For any $z\in \mbS^n$ and  $t>0$, the following statements hold.
\begin{itemize}
    \item[(i)] $F^o(tz)=tF^o(z)$.
\item[(ii)] $\left<\Phi(z),z\right>=F(z)$.
\item[(iii)] $F^o(\Phi(z))=1$.
\item[(iv)] The following Cauchy-Schwarz inequality holds:
\begin{align}\label{Cau-Sch}\<x, z\>\le F^o(x)F(z).
\end{align}
\item[(v)]  The unit
  Wulff shape $\mcW_F$ can be interpreted by $F^o$ as $$\mcW_F=\{x\in\mbR^{n+1}|F^o(x)=1\}.$$
\end{itemize}
\end{proposition}
A Wulff shape of radius $r$ centered at $x_0\in \rr^{n+1}$ is given by \begin{align*}
\mcW_{r_0}(x_0)=\{x\in\mbR^{n+1}|F^o(x-x_0)=r_0\}.
\end{align*}

Let $\S\subset\rr^{n+1}_+$ be a $C^2$ hypersurface with $\p\S\subset \p\rr^{n+1}_+$, which encloses a bounded domain $\O$. Let $\nu$ be the unit normal of $\S$ pointing outward $\O$. 
The anisotropic normal of $\S$ is given $$\nu_F=\Phi(\nu)=\nabla F(\nu)+F(\nu)\nu,$$
and the anisotropic principal curvatures $\{\kappa_i^F\}_{i=1}^n$ of $\S$ are given by the eigenvalues of the anisotropic Weingarten map $$d\nu_F=A_F(\nu)\circ d\nu: T_p\S\to T_p\S.$$ The eigenvalues are real since $(A_F)$ is positive definite and symmetric.
For $r\in \{1,\cdots, n\}$, the (normalized) $r$-th anisotropic mean curvature is defined by
$$H_r^F=\frac{1}{\binom{n}{r}}\sigma_r^F,$$ where $\sigma_r^F$ be the $r$-th elementary symmetric function on the anisotropic principal curvatures $\{\kappa_i^F\}_{i=1}^n$, namely,
\begin{align*}
    \sigma_r^F=\sum_{1\leq i_1<\cdots<i_r\leq n}\kappa^F_{i_1}\cdots\kappa^F_{i_r},
\end{align*}
In particular, $H^F=\sigma_1^F$ is the anisotropic mean curvature and $H_1^F$ the normalized anisotropic mean curvature.
Alternatively, the $r$-th anisotropic mean curvature $H^F_r$ of $\S$ can be defined through the following identity:
\begin{align}\label{defn-Pn(t)}
    \mcP_n(t)=\prod_{i=1}^n(1+t\kappa^F_i)=\sum_{i=0}^n\binom{n}{i}H^F_it^i
\end{align}
for all real number $t$.

It is easy to check that the anisotropic principal curvatures of $\mcW_r(x_0)$ are $\frac 1{r}$, since 
\begin{align}\label{normal}
    \nu_F(x)=\frac{x-x_0}{r},  \quad\hbox{ on } \mcW_r(x_0).\end{align}
For the convenience of the reader, we provide a proof of \eqref{normal}. We first consider the unit Wulff shape, $\mcW_F.$
 Since the unit normal vector $\nu$ of $\mcW_F$ at $x \in \mcW_F$ is given by $\Phi^{-1}(x)$, then the anisotripic normal is just $x$. For the general case, one use a translation and a scaling.


A truncated Wulff shape is a part of a Wulff shape  cut by a hyperplane, say $\{x_{n+1}= 0\}$. Namely, it is an intersection of a Wulff shape and $\R^{n+1}_+$. As mentioned above, it was used by Winterbottem \cite{Win67}. At a first glimpse, it is not very easy to image why the hyperplane intersects 
a Wulff shape at a ``constant angle" as in the isotropic case, namely \eqref{acbc} holds. It follows from
$$ \langle \nu_F, -E_{n+1} \rangle = \langle \frac {x-x_0} r, -E_{n+1} \rangle= \langle \frac {x_0} r, -E_{n+1} \rangle,  $$
which is a constant.

\begin{remark}\label{rem-necc}
\normalfont
The boundary condition $\<\Phi(\nu), -E_{n+1}\>=\omega_0$ implies $\om_0\in (-F(E_{n+1}),F(-E_{n+1}))$. Indeed, by the Cauchy-Schwarz inequality \eqref{Cau-Sch}, 
 \begin{align}\label{om-cond}
 -F(E_{n+1})=-F(E_{n+1})F^o(\Phi(\nu))\leq  \<\Phi(\nu),-E_{n+1}\>\leq F^o(\Phi(\nu))F(-E_{n+1})=F(-E_{n+1}).
 \end{align}
Since $\S$ is embedded, $\S$ intersects $\p\rr_+^{n+1}$ transversely. It follows that equality in \eqref{om-cond} cannot hold.
Therefore, $\om_0\in(-F(E_{n+1}),F(-E_{n+1}))$ is a necessary condition for anisotropic $\om_0$-capillary hypersurfaces. 
\end{remark}

From our work  in this paper, one can in fact  introduce a notion of ``anisotropic contact angle" as follows, which is natural generalization of
the contact angle in the isotropic case.
We define $\theta :\p \S \to (0, \pi)$ by
\begin{eqnarray*}
   - \cos \theta = \left\{ \begin{array}{cl}
    F(E_{n+1})^{-1}  \langle \nu_F, -E_{n+1} \rangle, & \quad \hbox{ if }\langle \nu_F, -E_{n+1} \rangle <0 ,\\
    0, &\quad \hbox{ if } 
    \langle \nu_F, -E_{n+1} \rangle=0, \\
     F(-E_{n+1})^{-1}  \langle \nu_F, -E_{n+1} \rangle, & \quad \hbox{ if }\langle \nu_F,- E_{n+1} \rangle >0.\\
    \end{array}
    \right.
    \end{eqnarray*}
If $\theta=\pi/2$, or equivalently  $\langle \nu_F, -E_{n+1} \rangle =0$, we call that the anisotropic hypersurface intersects $\p \R^{n+1}_+$
perpendicularly, or it is a free boundary anisotropic hypersurface.

\

\section{Minkowski type formula and Heintze-Karcher type inequality}\label{Sec-3}

\subsection{Minkowski type formula}
To prove the Minkowski type formula, we need the following structural lemma for compact hypersurfaces in $\rr^{n+1}$ with boundary, which is well-known and widely used, see for example \cite{AS16,JXZ22}.
	\begin{lemma}\label{Lem-structural}
	Let $\S\subset \rr^{n+1}$ be a compact hypersurface with boundary. Then it holds that
		\begin{align}\label{eq-structural}
			n\int_{\S}\nu dA=\int_{\p \S}\left\{\left<x,\mu\right>\nu-\left<x,\nu\right>\mu\right\}ds.
		\end{align}
		
	\end{lemma}
In the paper we denote  $\mu$  the unit outward co-normal of $\p\S$ in $\S$.
Recall $E^{F}_{n+1}$ defined in \eqref{eq-E-F}. It is easy to check that
\begin{align}\label{eq-E-F-1}
\<E^{F}_{n+1},E_{n+1} \>=1.
\end{align}


\

\noindent{\it Proof of \cref{Prop-Minkowski0}.} We first prove \eqref{formu-Minkowski_1}.
    We begin by introducing the following $C^1$ vector field along $\S$:
    \begin{align*}
        X_F(x)=F(\nu(x))x-\left<x,\nu(x)\right>\nu_F(x).
    \end{align*} Observe that $X_F$ is indeed a tangential vector field along $\S$, since
    \begin{align*}
        \left<X_F,\nu\right>=F(\nu)\left<x,\nu\right>-\left<x,\nu\right>\left<\nu_F,\nu
        \right>=0.
    \end{align*}
Notice also that along $\S$ we have
\begin{align}\label{div-X}
    {\rm div}_\S(X_F)
    =&nF(\nu)+\<d\nu(\nabla F),x\>-\left<x,d\nu(\nu_F^T)\right>-\left<x,\nu\right>H^F=nF(\nu)- H^F\left<x,\nu\right>,
\end{align}
where ${\rm div}_\S$ is the divergence on $\S$.
In the second equality, we have used the self-adjointness of $d\nu$. Here $\nu^T_F$ and $x^T$ denote the tangential projection on $\S$ of $\nu_F$ and $x$ respectively. In particular, $\nu^T_F=\nu_F-\<\nu_F, \nu\>\nu=\nabla F(\nu)$.
On  one hand, integrating \eqref{div-X} along $\S$ and using the divergence theorem, we find
\begin{align}\label{eq-Minko-1}
    \int_\S nF(\nu)-H^F(x)\left<x,\nu\right>dA=\int_{\p\S}F(\nu)\left<x,\mu\right>-\left<x,\nu\right>\left<\nu_F,\mu\right>ds.
\end{align}
On the other hand, by \eqref{eq-structural} we have
\begin{align}\label{eq-Minko-2}
    -n\om_0\int_\S\left<\nu, E^F_{n+1}\right>dA=\int_{\p\S}(-\left<x,\mu\right>\left<\nu,E^F_{n+1}\right>\om_0+\left<x,\nu\right>\left<\mu,E^F_{n+1}\right>\om_0) ds.
\end{align}
It is easy to see that  at any $x\in\p\S\subset\p\mbR^{n+1}_+$
\begin{align*}
    E_{n+1}=\left<\nu,E_{n+1}\right>\nu+\left<\mu,E_{n+1}\right>\mu,
\end{align*}
and  hence we have 
\begin{align}
    -\om_0&=\left<E_{n+1},\nu_F\right>=\left<\nu,E_{n+1}\right>F(\nu)+\left<\mu,E_{n+1}\right>\left<\nu_F,\mu\right>,\label{eq-Minko-om}\\
    0&=\left<E_{n+1},x\right>=\left<\nu,E_{n+1}\right>\left<x,\nu\right>+\left<\mu,E_{n+1}\right>\left<x,\mu\right>.\label{eq-Minko-0}
    \end{align}
    Moreover by \eqref{eq-E-F-1} we have 
    \begin{align}
    1=\<E^F_{n+1},E_{n+1}\>=\<\nu,E_{n+1}\>\<E^F_{n+1},\nu\>+\<\mu,E_{n+1}\>\<E^F_{n+1},\mu\>.\label{eq-Minko-F}
\end{align}
It yields
\begin{align*}
&-\left<x,\mu\right>\left<\nu,E^{F}_{n+1}\right>\om_0+\left<x,\nu\right>\left<\mu,E^{F}_{n+1}\right>\om_0\\
   =&\left<x,\mu\right>\left<\nu,E^{F}_{n+1}\right>\left<\nu,E_{n+1}\right>F(\nu)-\left<x,\nu\right>\left<\mu,E^{F}_{n+1}\right>\left<\nu,E_{n+1}\right>F(\nu)\\
   &+\left<x,\mu\right>\left<\nu,E^{F}_{n+1}\right>\left<\mu,E_{n+1}\right>\left<\nu_F,\mu\right>-\left<x,\nu\right>\left<\mu,E^{F}_{n+1}\right>\left<\mu,E_{n+1}\right>\left<\nu_F,\mu\right>\\
   =&\left<x,\mu\right>\left<\nu,E^{F}_{n+1}\right>\left<\nu,E_{n+1}\right>F(\nu)+\left<x,\mu\right>\left<\mu,E_{n+1}\right>\left<\mu,E^{F}_{n+1}\right>F(\nu)\\
   &-\left<x,\nu\right>\left<\nu,E_{n+1}\right>\left<\nu,E^{F}_{n+1}\right>\left<\nu_F,\mu\right>+\left<x,\nu\right>\left<\mu,E^{F}_{n+1}\right>\left<\mu,E_{n+1}\right>\left<\nu_F,\mu\right>\\
   =&F(\nu)\left<x,\mu\right>-\left<x,\nu\right>\left<\nu_F,\mu\right>,
\end{align*}
where we have used \eqref{eq-Minko-om} in the first equality,  \eqref{eq-Minko-0} in the second equality and \eqref{eq-Minko-F} in the last one. In particular, this identity, together with \eqref{eq-Minko-1} and \eqref{eq-Minko-2}, implies 
\begin{align*}
    \int_\S nF(\nu)-H^F(x)\left<x,\nu\right>dA=-n\int_\S\left<\nu,\om_0 E^F_{n+1}\right>dA,
\end{align*}
which is \eqref{formu-Minkowski_1}.

Next we prove \eqref{formu-Minkowski} for general $r$ by using \eqref{formu-Minkowski_1} as in \cite{JWXZ22}.
Consider a family of hypersurfaces  $\S_t$ with boundary for small $t>0$,
defined by 
		$$\varphi_t(x)=x+t(\nu_F(x)+\om_0E^{F}_{n+1})\quad x\in\S.$$
	We claim that $\S_t$ is also an anisotropic $\om_0$-capillary hypersurface in $\mbR^{n+1}_+$.
	On one hand, the $\om_0$-capillarity condition and \eqref{eq-E-F-1} yield that for any $x\in\p\S$,
\begin{align*}
    \left<\nu_F(x)+\om_0E^{F}_{n+1},-E_{n+1}\right>=\om_0-\om_0=0.
\end{align*}
Hence, $\varphi_t(x)\in \p\rr_+^{n+1}$ for $x\in \p\S$ which means $\p\S_t\subset \p\rr_+^{n+1}$.
 On the other hand, denoting by $e^F_i$ an anisotropic principal vector at $x\in\S$ corresponding to $\kappa_i^F$ for $i=1,\cdots, n$, we have
		\begin{align}\label{eq-principaldirections}
	 (\varphi_t)_\ast(e^F_i)=(1+t\kappa^F_i)e^F_i,\quad i=1,\ldots,n.
	\end{align}
We see from \eqref{eq-principaldirections} that $\nu^{\S_t}(\varphi_t(x))=\nu(x)$,
and in turn $\nu_F^{\S_t}(\varphi_t(x))=\nu_F(x)$.
Here $\nu^{\S_t}$ and $\nu_F^{\S_t}$ denote the outward normal and anisotropic normal to $\S_t$ respectively.
	In view of this, we have $$\left<\nu_F^{\S_t}(\varphi_t(x)),-E_{n+1}\right>=\left<\nu_F(x),-E_{n+1}\right>=\om_0.$$ Therefore, $\S_t$ is also an anisotropic $\om_0$-capillary hypersurface in $\mbR^{n+1}_+$ and hence \eqref{formu-Minkowski_1}
 holds for $\S_t$ for any small $t$.
 Exploiting \eqref{formu-Minkowski_1} for every such $\S_t$, we find
 \begin{align}\label{eq-PropMin-0} \int_{\S_t=\varphi_t(\S)}
 \left(F(\nu_t)+\om_0\left<\nu_t,E^F_{n+1}\right>\right)-H_1^F(t)\mid_y\left<y,\nu_t\right>
 dA_t(y)=0.
	\end{align}
By \eqref{eq-principaldirections}, the tangential Jacobian of $\varphi_t$ along $\S$ at $x$ is just
\begin{align}\label{eq-PropMin-1}
    {\rm J}^\S\varphi_t(x)=\prod_{i=1}^n(1+t\kappa_i^F(x))=\mcP_n(t),
\end{align}
where $\mcP_n(t)$ is the polynomial defined in \eqref{defn-Pn(t)}.
Moreover, using \eqref{eq-principaldirections} again, we see that
the corresponding anisotropic principal curvatures are given by
\begin{align}\label{eq-PropMin-2}
    \kappa^F_i(\varphi_t(x))=\frac{\kappa^F_i(x)}{1+t\kappa^F_i(x)}.
\end{align}
Hence fix $x\in\S$, the anisotropic mean curvature of $\S_t$ at  $\varphi_t(x)$, say $H^F(t)$, is given by
\begin{align}
\label{29}
    H^F(t)
    =\frac{\mcP'_n(t)}{\mcP_n(t)}
    =\frac{\sum_{i=0}^{n}i\binom{n}{i} H^F_it^{i-1}}{\mcP_n(t)},
\end{align}
where $H^F_i=H^F_i(x)$ is the $i$-th mean curvature of $\S$ at $x$.

Using the area formula, \eqref{eq-PropMin-1} and \eqref{29}, we find from \eqref{eq-PropMin-0} that
\begin{align*}
    \int_\S n\left(F(\nu)+\om_0\left<\nu,E_{n+1}^F\right>\right)\mcP_n(t)-t\left(\left<x,\nu_F(x)\right>+\om_0\left<x,E_{n+1}^F\right>\right)\mcP'_n(t)-\mcP'_n(t)\left<x,\nu\right>dA_x=0.
\end{align*}
As the left hand side in this equality is a polynomial in the time variable $t$, this shows that all its coefficients vanish, and hence a direct computation yields \eqref{formu-Minkowski}.
\qed

We remark that the definition of the family of capillary hypersurfaces $\S_t$ was inspired by \cite{JWXZ22}.
These are the parallel hypersurfaces in the case of capillary boundary.

\begin{remark}\label{rem-1}
\normalfont
If we replace $E^F_{n+1}$ by $E_{n+1}$ in the proof, every step above is valid and we achieve that
\begin{align*}
    \int_\S H^F_{r-1}\left(F(\nu)+\om_0\left<\nu, E_{n+1}\right>\right)-H^F_r\left<x,\nu\right>dA=0.
\end{align*} 
Alternatively, we can prove directly that 
\begin{align*}
    \int_\S H^F_{r-1}\left<\nu, E^F_{n+1}\right>dA=\int_\S H^F_{r-1}\left<\nu, E_{n+1}\right>dA.
\end{align*}
Since we do not need it in this paper, we omit the proof.
\end{remark}

\subsection{Heintze-Karcher type inequality}

To prove the Heintze-Karcher type inequality, we need the following key proposition, which amounts to be an anisotropic angle comparison principle. It is clear that in the isotropic case it is trivial. However in the anisotropic case it is non-trivial.

\begin{proposition}\label{Prop-angle-compare}
 Let $x,z\in\mbS^n$ be two distinct points and $y\in\mbS^n$ lie in a length-minimizing geodesic joining $x$ and $z$ in $\mbS^n$,
then we have $$\left<\Phi(x),z\right>\leq\left<\Phi(y),z\right>.$$ 
Equality holds if and only if $x=y$.
\end{proposition}
\begin{proof}
We denote  $d_0=d_{\mbS^n}(x,z)$ and $d_1=d_{\mbS^n}(x,y)$, where $d_{\mbS^n}$ denotes the intrinsic distance on $\mbS^n$.
If $y\neq x$, clearly $0<d_1\leq d_0$.
Let $\gamma:[0,d_0]\to \mbS^n$ be the arc-length parameterized geodesic with $\gamma(0)=x$, $\gamma(d_0)=z$.  Considering the following function $$f=\<\Phi(\gamma(t)),z\>, \quad t\in [0,d_0].$$ We have
\begin{align}\label{eq-prop}
  & \<\Phi(y),z\>-\<\Phi(x),z\>=f(d_1)-f(0)\nonumber\\ =& \int_0^{d_1} \left<\frac{d}{dt}\Phi(\gamma(t)),z\right>dt=\int_0^{d_1}\<D_{\dot{\gamma}(t)}\Phi(\gamma(t)), z\> dt,
\end{align}
where $D$ is the Euclidean covariant derivative.
Since $\gamma$ is length-minimizing, it is easy to see that  $$\<\dot{\g}(t), z\>\ge 0, \quad \forall t\in (0, d_0).$$ Thus $z$ can be expressed as $z=\sin s\dot{\gamma}(t)+\cos s\gamma(t)$ with some $s\in(0,\pi)$. It follows that
\begin{align*}
  \<D_{\dot{\gamma}(t)}\Phi(\gamma(t)), z\>=\sin s (\nabla^2 F+F I)(\dot{\gamma}(t),\dot{\gamma}(t)),
\end{align*}
Since  $(\nabla^2 F+F I)>0$,
we get $\<D_{\dot{\gamma}(t)}\Phi(\gamma(t)), z\>>0$ for any $t\in(0,d_1)$.
This fact, together with \eqref{eq-prop}, leads to the assertion.
\end{proof}

We note that when $y=z$, \cref{Prop-angle-compare} is nothing but the Cauchy-Schwarz inequality \eqref{Cau-Sch}, since one readily observes from \cref{basic-F}(ii)(iii) that
\begin{align*}
    \left<\Phi(x),z\right>\leq F^o(\Phi(x))F(z)=\left<\Phi(z),z\right>.
\end{align*}
The Cauchy-Schwarz inequality \eqref{Cau-Sch} also impies the following property. 
\begin{proposition}\label{Prop-non-negative}
For $\om_0\in\left(-F(E_{n+1}),F(-E_{n+1})\right)$, there holds
\begin{align*}
F(z)+\om_0\left<z,E^F_{n+1}\right>>0,\quad\text{for any }z\in\mbS^n.
\end{align*}
\end{proposition}
\begin{proof}
It is clear that we need not to consider the case $\omega_0=0$. If $\omega_0<0$, we only need to consider the points $z$ satisfying $\<z,E^F_{n+1}\>>0$. At any such $z$, since $\omega_0>-F(E_{n+1})$, we have
\begin{align*}
F(z)+\om_0\left<z, E^F_{n+1}\right>
>&F(z)-F(E_{n+1})\left<z, \frac{\Phi(E_{n+1})}{F(E_{n+1})}\right>\nonumber\\
\geq& F(z)-F(z)F^o(\Phi(E_{n+1}))=0,
\end{align*}
where we have used the Cauchy-Schwarz inequality for the second inequality.


For the case $\omega_0>0$, 
 we just need to consider the points $z$ such that $\<z,E^F_{n+1}\><0$.
 Since $\omega_0<F(-E_{n+1})$, using the Cauchy-Schwarz inequality again, we find
\begin{align*}
F(z)+\om_0\left<z, E^F_{n+1}\right>
>& F(z)+F(-E_{n+1})\left<z,-\frac{\Phi(-E_{n+1})}{F(-E_{n+1})}\right>\nonumber\\
\geq&F(z)-F(z)F^o(\Phi(-E_{n+1}))=0.
\end{align*}
The proposition is thus proved.
\end{proof}

Now we can start to prove \cref{Thm-HK}.

\begin{proof}[Proof of \cref{Thm-HK}]
    Let $\S\subset\bar\mbR_+^{n+1}$ be an anisotropic capillary hypersurface satisfying \eqref{capi-ass}.
For any $x\in \S$, let $\kappa^F_i(x)$ be the  anisotropic principal curvature and $e^F_i(x)$ be the corresponding anisotropic principal vector of $\S$ at $x$ such that $|e^F_1\wedge e^F_2\wedge\cdots \wedge e^F_n|=1$. Since $\S$ is strictly anisotropic mean convex,
\begin{align*}
    \max_i{\kappa^F_i(x)}\ge \frac1n H^F(x)>0, \hbox{ for }x\in \S.
\end{align*}
We define 
\begin{align*}
    Z=\left\{(x,t)\in\S\times\mbR:0<t\leq\frac{1}{\max{\kappa^F_i(x)}}\right\},
   \end{align*} 
   and
   \begin{align}
    &\zeta_F: Z\to \rr^{n+1},\\
    &\zeta_F(x,t)=x-t\left(\nu_F(x)+\om_0 E^{F}_{n+1}\right).\label{parallel}
\end{align}

\noindent{\bf Claim:} $\O\subset\zeta_F(Z)$.

Recall $\mcW_r(x_0)$ is the Wulff shape centered at $x_0$ with radius $r$. 
For any $y\in\O$, we consider a family of Wulff shapes 
$\{ \mcW_r(y+r\om_0 E^F_{n+1})\}_{r\geq0}$.
Since $y\in\O$ is an interior point,  we definitely have $\mcW_r(y+r\om_0 E^F_{n+1})\subset\O$ for $r$ small enough.
On the other hand, by the assumption $-F(E_{n+1})<\om_0< F(-E_{n+1})$, the definition \eqref{eq-E-F} of $E_{n+1}^F$ and \cref{basic-F}(i)(iii), it is easy to see that  \begin{align*}
    F^o(-\om_0 E^F_{n+1})< 1.
\end{align*} 
It follows that
\begin{align*}
   F^o\big( y-(y+r\omega_0 E^F_{n+1})\big)=r F^o(-\omega_0 E^F_{n+1})< r,
\end{align*}
which implies for any small $r>0$, $y$ is always in the domain bounded by the Wulff shape $\mcW_r(y+r\om_0 E^{F}_{n+1})$.
Hence $\mcW_r(y+r\om_0E^F_{n+1})$
must touch $\S$ as we increase the radius $r$.
Consequently, for any $y\in\O$, there exists $x\in\S$ and $r_y>0$, such that $\mcW_{r_y}(y+r_y\om_0 E^F_{n+1})$ 
touches $\S$ for the first time, at some point $x\in \S$.
In terms of the touching point, only the following two cases are possible:

\textbf{Case 1.} $x\in\mathring{\Sigma}$. 

In this case, since $x\in\mathring{\S}$, 
the Wulff shape $\mcW_{r_y}(y+r_y\om_0 E^F_{n+1})$ is tangent to $\S$ at $x$ from the interior.
Hence 
\begin{align}\label{eq-nu} \nu(x)=\nu^{\mcW}(x),
\end{align} 
where $\nu^{\mcW}$  denotes the outward unit normal  of $\mcW_{r_y}(y+r_y\om_0 E^F_{n+1})$.
Moreover, since the touching of $\mcW_{r_y}(y+r_y\om_0 E^F_{n+1})$ with $\S$ is from interior, we see that 
 \begin{align}\label{dnu-eq}
    d\nu\le d\nu^\mcW,
\end{align} 
in the sense that
the coefficient matrix of the difference of two classical Weingarten operators $d\nu-d\nu^\mcW$ is semi-negative definite. It follows from \eqref{eq-nu} and \eqref{dnu-eq} that that  \begin{align}\label{an-dnu-eq} A_F(\nu)\circ d\nu\le A_F(\nu^\mcW) \circ d\nu^\mcW.
\end{align}
Since the  anisotropic principal curvatures of $\mcW_{r_y}(y-r_y\om_0 E^F_{n+1})$ are equal to $\frac 1{r_y}$, we see from \eqref{an-dnu-eq} that
\begin{align*}
\max_{1\leq i\leq n}{\kappa^F_{i}(x)}\le\frac 1{r_y}.
\end{align*} 
Invoking the definition of $Z$ and $\zeta_F$, we find that $y\in\zeta_F(Z)$ in this case.


\textbf{Case 2.} $x\in\p\S$. 

We will rule out this case by the capillarity assumption \eqref{capi-ass}.
Let $\nu_F^\mcW(x)$ be the outward anisotropic normal to $\mcW_{r_y}(y+r_y\om_0 E^F_{n+1})$. It is easy to see that $$\nu_F^\mcW(x)=\Phi(\nu^\mcW(x))=\frac{x-(y+r_y\om_0 E^{F}_{n+1})}{r_y}.$$ 
Recall that $y$ lies in the interior of $\Om$. Thus $\left<y,E_{n+1}\right>>0$. On one hand, in view of  \eqref{capi-ass} and \eqref{eq-E-F-1} we have
\begin{align}\label{angle-contr}
    \left<\nu_F^\mcW(x),-E_{n+1}\right>=1/r_y\left<y,E_{n+1}\right>+\om_0\<E^{F}_{n+1},E_{n+1}\> >\om_0\geq\om(x)=\left<\nu_F(x),-E_{n+1}\right>.
\end{align}
On the other hand,
since the Wulff shape $\mcW_{r_y}(y+r_y\om_0 E^F_{n+1})$ touches $\S$ from the interior, we have
\begin{align*}
\left\<\nu(x),-E_{n+1}\right>\geq\left<\nu^\mcW(x),-E_{n+1}\right>.
\end{align*}
Since $\nu$, $\nu^\mcW$ and $-E_{n+1}$ lie on the two-plane orthogonal to $T_x(\p\S)$, we see that
 $\nu$ lies actually in the geodesic joining $\nu^\mcW$ and $-E_{n+1}$ in $\mbS^n$.
It then follows from the angle comparison principle \cref{Prop-angle-compare} that 
\begin{align*}
\left\<\Phi(\nu(x)),-E_{n+1}\right>\geq\left<\Phi(\nu^\mcW(x)),-E_{n+1}\right>.
\end{align*}
This is a contradiction to \eqref{angle-contr}.
The {\bf Claim} is thus proved.
 
By a simple computation, we find
\begin{align*} 
\p_t \zeta_F(x, t)&=-\left(\nu_F(x)+\om_0 E^F_{n+1}\right), \\ 
D_{e^F_i} \zeta_F(x, t)&=\left(1-t\kappa^F_i(x)\right)e^F_i(x).
\end{align*}
Thanks to \cref{Prop-non-negative},
a classical computation yields that the tangential Jacobian of $\zeta_F$ along $Z$ at $(x,t)$ is just
$${\rm J}^Z\zeta_F(x,t)= (F(\nu)+\om_0\<\nu, E^F_{n+1}\>)\prod_{i=1}^n(1-t\kappa^F_i).$$
By virtue of the fact that $\O\subset\zeta_F(Z)$, the area formula yields
\begin{align*}
    \vert\O\vert\leq\vert\zeta_F(Z)\vert
    \leq&\int_{\zeta_F(Z)}\mcH^0(\zeta_F^{-1}(y))dy
    =\int_Z {\rm J}^Z\zeta_F d\mcH^{n+1}\notag\\
    =&\int_\Sigma dA\int_0^{\frac{1}{\max\left\{\kappa^F_i(x)\right\}}}\left (F(\nu)+\om_0\left<\nu,E^F_{n+1}\right>\right)\prod_{i=1}^n(1-t\kappa^F_i(x))dt.
\end{align*}
By the AM-GM inequality,
and the fact that $\max\left\{\kappa^F_i(x)\right\}_{i=1}^n\geq \frac1n H^F(x)$, we obtain
\begin{align*}
    \vert\O\vert
    \leq&\int_\S dA\int_0^{\frac{1}{\max\left\{\kappa^F_i(x)\right\}}} \left(F(\nu)+\om_0\<\nu, E^F_{n+1}\>\right)\left(\frac{1}{n}\sum_{i=1}^n\left(1-t\kappa^F_i(x)\right)\right)^n dt\notag\\
    \leq&\int_\S \left(F(\nu)+\om_0\<\nu, E^F_{n+1}\>\right)dA\int_0^{\frac{n}{H^F(x)}} \left(1-t\frac{H^F(x)}{n}\right)^n dt\notag\\
    =&\frac{n}{n+1}\int_\S \frac{F(\nu)+\om_0\<\nu, E^F_{n+1}\>}{H^F} dA,
\end{align*}
which gives \eqref{eq-HK0}.

If equality in \eqref{eq-HK0} holds, then from the above argument, we see  $\kappa_1^F(x)=\ldots=\kappa_n^F(x)$ for all $x\in\S$.
It follows from \cite[Lemma 2.3]{HLMG09}  that $\S$ must be a part of a Wulff shape  $\mcW_{r_0}(x_0)$ for some $r_0$ and some point $x_0$. Hence $H^F$ is a constant $\frac{n}{r_0}$ and since $\nu_F(x)=\frac{x-x_0}{r_0}$, we get for $x\in \p\S$,
$$\omega(x)=\<\nu_F(x), -E_{n+1}\>=\frac{1}{r_0}\<x_0,E_{n+1}\>:=\tilde{\omega}_0,$$
which is a constant.

From the equality in  Heintze-Karcher inequality and the Minkowski-type formula \eqref{formu-Minkowski}, taking into account that $H_F$ is a constant, we deduce that
$$\omega_0\int_\S \left<\nu,E_{n+1}^F\right\>dA=\tilde{\omega}_0\int_\S \left<\nu,E_{n+1}^F\right\>dA.$$
By the divergence theorem and \eqref{eq-E-F-1}, $$\int_\S\left<\nu,E_{n+1}^F\right\>dA=|\p\O\cap \p\rr_+^{n+1}|\neq 0.$$
It follows that $\tilde{\omega}_0=\omega_0$ which means $\S$ is a $\om_0$-capillary Wulff shape.

Conversely, for any $\om_0$-capillary Wulff shape, we can see easily from the Minkowski-type formula \eqref{formu-Minkowski} and  the fact of constant anisotropic mean curvature that equality holds in \eqref{eq-HK0}.
This completes the proof.
\end{proof}

\begin{remark}\label{rem-2}
\normalfont
We may use in the proof another foliation of Wulff shapes $\{\mcW_r(y+r\o_0E_{n+1})\}_{r\ge 0}$. To ensure that $\mcW_r(y+r\o_0E_{n+1})$ intersects with $\S$ for large $r$, we need to assume \begin{align}\label{assum-omega0}\o_0\in \left(-\frac{1}{F^o(E_{n+1})}, \frac{1}{F^o(-E_{n+1})}\right).\end{align}
We can follow the proof to achieve that
     \begin{align}\label{eq-HK0-1}
     \int_\S\frac{F(\nu)+\om_0\left<\nu,E_{n+1}\right>}{H^F}dA\geq\frac{n+1}{n}\vert\O\vert.
    \end{align}
under the assumption \eqref{assum-omega0}. On the other hand, by virtue of the Cauchy-Schwarz inequality, we see that \eqref{assum-omega0} is in general more restrictive than the natural assumption $\o_0\in (-{F(E_{n+1}), {F(-E_{n+1})}}).$ This is the reason why we introduce $E^F_{n+1}$.
\end{remark}

\noindent{\it Proof of \cref{coro1}}. 
It is clear that any closed hypersurface can be seen as a capillary surface in a half space  with a empty boundary.  For any $e\in \mathbb{S}^n$ we can see $e$ as $E_{n+1}$ and apply \cref{Thm-HK}.

First  consider $\om_0 \in (-F(E_{n+1}), 0)$. Together with the definition of $E_{n+1}^F$ \eqref{eq-HK0} gives us 
    \begin{align*}
     \int_\S\frac{F(\nu)}{H^F}dA\geq\frac{n+1}{n}\vert\O\vert 
     - \om_0  \int_\S \frac {\langle \nu, E^F_{n+1} \rangle}{H^F} \, dA
     =\frac{n+1}{n}\vert\O\vert 
     - \frac{\om_0}{F(E_{n+1})}  \int_\S \frac {\langle \nu,  \Phi(E_{n+1}) \rangle}{H^F} \, dA.
    \end{align*}
   It follows  that
\begin{align*}
     \int_\S\frac{F(\nu)}{H^F}dA\geq\frac{n+1}{n}\vert\O\vert 
     + \max\left\{0, \int_\S \frac {\langle \nu,  \Phi(E_{n+1}) \rangle}{H^F} \, dA \right\}.
    \end{align*}
   Then we consider $\om_0 \in (0, F(-E_{n+1}))$. Similarly, in this case \eqref{eq-HK0} gives us 
  \begin{align*}
     \int_\S\frac{F(\nu)}{H^F}dA &\geq
      \frac{n+1}{n}\vert\O\vert 
     +\max\left\{0, \int_\S \frac {\langle \nu,  \Phi(-E_{n+1}) \rangle}{H^F} \, dA.\right\}.
    \end{align*}
This completes the proof of \eqref{eq1.6}. 
\qed

\begin{remark}
	\normalfont
When $F$ is even, i.e., $F(x)=F(-x)$, we have
     \begin{align*}
     \int_\S\frac{F(\nu)}{H^F}dA\geq
     \frac{n+1}{n}\vert\O\vert +  \max_{e\in \mathbb{S}^n} \int_\S \frac{ \langle \nu,\Phi (e) \rangle }{H^F}\, dA \, \left(\ge \frac{n+1}{n}\vert\O\vert\right).
    \end{align*}
Since in this case, $\Phi(-E_{n+1})=-\Phi(E_{n+1})$. Hence either $\int_\S \frac {\langle \nu,  \Phi(E_{n+1}) \rangle}{H^F} \, dA\ge 0$ or $\int_\S \frac {\langle \nu,  \Phi(-E_{n+1}) \rangle}{H^F} \, dA\ge 0$.

\end{remark}
\

\section{Alexandrov type Theorem}\label{Sec-4}
We first prove a result on the existence of an elliptic point for an anisotropic capillary hypersurface.
\begin{proposition}\label{Prop-EllipticPoint}
Let $\om_0\in\left(-F(E_{n+1}),F(-E_{n+1})\right)$ and
let $\S\subset\mbR^{n+1}_+$ be a $C^2$ compact embedded anisotropic $\om_0$-capillary hypersurface, then $\S$ has at least one elliptic point, i.e. a point where all the anisotropic principal curvatures are positive.
\end{proposition}
\begin{proof}
   We fix a point $y\in {\rm int}(\p\O\cap\p\mbR_+^{n+1})$. Consider the family of  Wulff shapes $\mcW_r(y+r\om_0 E^F_{n+1})$.
    Observe that for any $x\in\p\S$ and any $r>0$, there holds
    \begin{align}\label{eq-nuW-E}
    \left<\nu_F^\mcW(x),E_{n+1}\right>=\left<\frac{x-y}{r}+\om_0E^F_{n+1},E_{n+1}\right>=\om_0=\left<\nu_F(x),E_{n+1}\right>.
    \end{align}
Since $\S$ is compact, for $r$ large enough, $\S$ lies inside the domain bounded by the Wulff shape $\mcW_r(y+r\om_0 E^F_{n+1})$.
 Hence we can find the smallest $r$, say $r_0>0$, such that  $\mcW_{r_0}(y+r_0\om_0 E^F_{n+1})$ touches $\S$ at a first time at some $x_0\in\S$ from exterior.

If $x_0\in\mathring{\S}$, then $\S$ and $\mcW_{r_0}(y+r_0\om_0 E^F_{n+1})$ are tangent at $x$.
If $x_0\in\p\S$, from \eqref{eq-nuW-E},
we conclude again that  $\S$ and $\mcW_{r_0}((y+r_0\om_0 E^F_{n+1}))$ are tangent at $x$.
In both cases, by a similar argument as in the proof of \cref{Thm-HK}, we have that the anisotropic principal curvatures of $\S$ at $x_0$ are larger  than or equal to $\frac{1}{r_0}$. 
\end{proof}


\begin{proof}[Proof of \cref{Thm-Alexandrov0} and \cref{Thm-Alexandrov1}]
We begin by recalling that $\om_0\in\left(-F(E_{n+1}),F(-E_{n+1})\right)$ ensures the non-negative of $F(\nu)+\om_0\left<\nu,E^F_{n+1}\right>$ pointwisely along $\S$, thanks to \cref{Prop-non-negative}.

On one hand,
by virtue of \cref{Prop-EllipticPoint} and G\"arding's argument \cite{Garding59} (see also \cite[Lemma 2.1]{HLMG09}),
    we know that $H^F_j$ are positive, for $j\le r$ and for any $x\in \S$.
    Applying \cref{Thm-HK} and using the Maclaurin inequality $H_1^F\ge H^F_r$ and the constancy of $H^F_r$,
    we have
    \begin{align}\label{eq-Alex-1} 
    (n+1)(H^F_r)^{1/r}\vert\O\vert
    \leq (H^F_r)^{1/r}\int_\S\frac{F(\nu)+\om_0\left<\nu,E^F_{n+1}\right>}{H_1^F}dA
    \leq\int_\S\left(F(\nu)+\om_0\left<\nu,E^F_{n+1}\right>\right)dA.
    \end{align}
    On the other hand,
    using the Minkowski-type formula \eqref{formu-Minkowski} and the Maclaurin inequality $H_{r-1}^F\ge (H^F_r)^{\frac{r-1}{r}}$, we have
    \begin{align*}
        0
        =&\int_\S H^F_{r-1}\left(F(\nu)+\om_0\left<\nu, E^F_{n+1}\right>\right)-H^F_r\left<x,\nu\right>dA\\
        \geq&\int_\S (H^F_r)^{\frac{r-1}{r}}\left(F(\nu)+\om_0\left<\nu, E^F_{n+1}\right>\right)-H^F_r\left<x,\nu\right>dA\\
        =&(H^F_r)^{\frac{r-1}{r}}\int_\S\left(F(\nu)+\om_0\left<\nu, E^F_{n+1}\right>\right)-(H^F_r)^{\frac{1}{r}}\left<x,\nu\right>dA\\
        =&(H^F_r)^{\frac{r-1}{r}}\left\{\int_\S\left(F(\nu)+\om_0\left<\nu, E^F_{n+1}\right>\right)dA-(n+1)(H^F_r)^{\frac{1}{r}}\vert\Om\vert\right\},
    \end{align*}
     where in the last equality we have used
     \begin{align*}
    (n+1)\vert\O\vert=\int_\O \div{x} \, dx=\int_\S\left<x,\nu\right>dA.
    \end{align*}
     Thus equality in \eqref{eq-Alex-1} holds, and hence $\S$ is an anisotropic $\om_0$-capillary Wulff shape. This completes the proof.
\end{proof}

\printbibliography

\end{document}